\documentclass{amsart}
\usepackage{cite}

\usepackage[english]{babel}
\usepackage{amssymb}
\usepackage{latexsym}
\usepackage{xypic}
\usepackage{multirow}
\usepackage[utf8]{inputenc}

\newtheorem{theorem}{Theorem}[section]
\newtheorem{lemma}[theorem]{Lemma}

\theoremstyle{definition}
\newtheorem{definition}[theorem]{Definition}
\newtheorem{example}[theorem]{Example}

\theoremstyle{remark}

\numberwithin{equation}{section}

\begin{document}

\title[Uniformization in higher dimensions]{\sc{On classical uniformization  theorems   for higher dimensional complex Kleinian groups }  }

\author{Angel Cano}
\address{ UCIM, Av. Universidad s/n. Col. Lomas de Chamilpa,
C.P. 62210, Cuernavaca, Morelos, M\'exico.}
\email{angelcano@im.unam.mx}

\author{Luis Loeza}
\address{IIT UACJ, Av. del Charro 610 Norte,  Partido Romero, C.P. 32310, Ciudad Ju\'arez, Chihuahua, M\'exico}
\email{  luis.loeza@uacj.mx}

\author{Alejandro Ucan-Puc}
\address{ UCIM, Av. Universidad s/n. Col. Lomas de Chamilpa,
C.P. 62210, Cuernavaca, Morelos, M\'exico.}
\email{manuel.ucan@im.unam.mx}

  \thanks{  Partially supported by grants of the PAPPIT's project IN101816 and CONACYT's  project 272169} 

\subjclass{Primary 37F99, 32Q, 32M Secondary 30F40, 20H10, 57M60, 53C}



\begin{abstract}
In this article we show that  Bers' simultaneous uniformization as well as  the Köebe's  retrosection theorem 
 are not longer true for discrete groups of projective transformations    acting on the complex projective space.
\end{abstract}

\maketitle

\section*{Introduction}
The  uniformization theorems of Riemann  surfaces plays a mayor role in one dimensional complex dynamics,
its study has its roots in the work of  Poincar\'e  and K\"oebe (see \cite{poincare,koebe_a,koebe_b,koebe_c}), 
the Riemann uniformization theorem asserts  that any simply connected Riemann surface is biholomophic
either   to  the sphere  $\hat{\mathbb{C}}$  the plane $\mathbb{C}$ or the disc $\mathbb{D},$ 
in 1910, K\"oebe  (see \cite{koebe} )``improved'' this ideas in his Retrosection theorem 
proving that for any closed Riemann surface there is a Schottky group such that the associated 
domain of discontinuity uniformizes the surface, later in 1960 Bers (see \cite{bers}) with his 
simultaneous uniformization theorem gives a ``generalisation'' of  the retrosection theorem, 
by asserting  that for any two closed  Riemann surfaces of the same genus, there is a quasi-Fuchsian group
that uniformizes the two surfaces.

Around 1990, J. Seade and A. Verjovsky (see \cite{SV2}) introduced the concept of complex Kleinian group 
as a discrete subgroup of $PSL(n,\mathbb{C})$ acting on the projective complex space with an open subset
where the action of the group is properly discontinuous, a very important family  of complex Kleinian groups
are the subgroups of isometries of the complex hyperbolic space (see \cite{goldman}). In view of this  groups
one natural question  arises, what about the uniformization of higher dimensional complex manifolds in terms
of complex Kleinian groups?, in this case we have that the geometry of higher dimensional complex manifolds
is much more diverse, to exemplify this let us consider the following facts:
  \begin{enumerate}
  \item  The complex manifold 
$\mathbb{P}^n_\mathbb{C}\times\mathbb{P}^n_\mathbb{C}$ is a simply connected complex manifold that
does not admit a complex projective structure, so if we want to ask  for the uniformization of higher
 dimensions in terms of complex kleinian groups  we must  require that the manifolds admit a  complex 
 projective structure (this is going to be a big  first difference with the one dimensional case).
 
 \item In \cite{Epstein} and \cite{thurston} they 
explain the Smilie's construction  of a torus with a complex projective structure that is not complete, that is 
the manifold cannot be realised as the quotient of an open set of the complex projective line and a discrete
subgroup that acts discontinuously on it. Once again,  when asking for uniformization results we must  assure 
that the manifolds that we  will study have  a complex projective structure which is complete.

\item  Last important fact to remember is that in the higher dimensional setting there are several  simply 
connected  domains of $\Bbb{P}^n_{\Bbb{C}}$ which are not biholomorphically  
equivalent. Moreover, some of this domains arise  as connected components in the equicontinuity region of 
complex  Kleinian groups, see \cite{CNS, CJP, CS1}.

  \end{enumerate} 
  
Taking in count this facts we see that is  fully expected that classical  uniformization 
theorems fail in the higher dimensional setting, in this article are able  to prove:

\begin{theorem} \label{t:1}
There is not an analogue of Bers simultaneous uniformization theorem for
groups of $PSL(n+1, \Bbb{C})$ acting on $\Bbb{P}^n_\Bbb{C}$, where $n\geq 2$.
\end{theorem}

This result was proven for $n=2$ in \cite{CS1}.  We also show  that the algebraic and geometric higher 
dimensional analog of the K\"oebe's retrosection theorem is false, here algebraic means that the fundamental
group of a manifold has a representation as a purely loxodromic free discrete group and geometric means that
the manifold can be realised as a quotient by a Schottky group. 

\begin{theorem} \label{t:2}
The geometric and algebraic version of   Köebe's retrosection is not longer true for
groups of $PSL(n+1, \Bbb{C})$ acting on $\Bbb{P}^n_\Bbb{C}$, where $n\geq 2$.
\end{theorem}

A weaker version of this result was essentially proven  for $n$ even   in \cite{cano}.

The paper is organised as follows: in section \ref{s:recall} we introduce the terminology  used along the 
article, in section \ref{s:slg} we introduce the notion of Schottky like groups and prove a technical lemma 
concerning its dynamic which will by useful in the last section, finally   in section \ref{s:pmt} we provide full proofs 
of the main results of this paper.

\section{Preliminaries}  \label{s:recall}
The complex projective space $\mathbb {P}^n_{\mathbb {C}}$
is defined as: $ \mathbb {P}^{n}_{\mathbb {C}}=(\mathbb {C}^{n+1}\setminus \{0\})/\Bbb{C}^* \,,$
where $\Bbb{C}^*$ acts by  the usual scalar multiplication. 
If $[\mbox{ }]:\mathbb {C}^{n+1}\setminus\{0\}\rightarrow
\mathbb {P}^{n}_{\mathbb {C}}$ is the quotient map, then a
non-empty set  $H\subset \mathbb {P}^n_{\mathbb {C}}$ is said to
be a projective subspace if there is a $\mathbb {C}$-linear
subspace  $\widetilde H$ such that $[\widetilde
H\setminus \{0\}]=H$.  In this article,  $e_1,\ldots, e_{n+1}$ will denote the standard basis for $\Bbb{C}^{n+1}$.
Given a set of points $P$   in $\mathbb{P}^{n}_{\mathbb{C}}$, we
define:
$$
Span( P)=\bigcap\{l\subset \mathbb{P}^n_{\mathbb{C}}\mid l \textrm{ is a projective subspace containing } P\}
.$$

 The group of projective automorphisms of $\mathbb{P}^{n}_{\mathbb{C}}$ is 
$PSL(n+1, \mathbb {C}) = GL({n+1}, \Bbb{C})/\Bbb{C}^*,$
where $\Bbb{C}^* $ acts by the usual scalar multiplication, $PSL(n+1, \mathbb {C})$ is a Lie group whose
elements are called projective transformations.
We denote   by $[[\mbox{  }]]: GL(n+1, \mathbb
{C})\rightarrow PSL(n+1, \mathbb {C})$    the quotient map. Given     $ \gamma  \in PSL(n+1, \mathbb
{C})$,  we  say that  $\widetilde \gamma  \in GL(n+1, \mathbb {C})$ is a  lift of $ \gamma  $ if 
$[[\widetilde  \gamma  ]]= \gamma  $.

Now consider the following Hermitian form $\langle\, ,\rangle:\Bbb{C}^{n+1}\rightarrow \Bbb{C}$, given by:
\[
\langle z,w\rangle=z_1\overline{w_1}+ \ldots+z_n\overline{w_n}-z_{n+1}\overline{w_{n+1}} 
\]
 We set
\[
U(1,n)=\{g\in GL(n+1,\Bbb{C}): \langle g(z), g(w)\rangle =\langle z, w\rangle \}.
\]
The respective projectivization $PU(1,n)$ 
preserves the unitary complex ball:
\[
\Bbb{H}^n_\Bbb{C}=\{[w]\in \Bbb{P}^n_{\Bbb{C}}\mid \langle w,w\rangle <0\}
\]

Elements in $PU(1,n)$ splits into 3 types according to the position of its fixed 
points in $\overline{\Bbb{H}^n_{\Bbb{C}}}$, more precisely:

\begin{definition}
 Let $\gamma\in PU(1,n)$, then $\gamma$ is called
\begin{enumerate}
  \item  elliptic if $\gamma$ has a fixed point in $\Bbb{H}^n_{\Bbb{C}}$.
  
  \item  parabolic  if $\gamma$ has exactly one fixed point in $\partial \Bbb{H}^n_{\Bbb{C}}$.
  
   \item  loxodromic  if $\gamma$ has exactly two fixed point in $\partial\Bbb{H}^n_{\Bbb{C}}$.
 
\end{enumerate}
\end{definition}

Given a group $\Gamma\subset PU(1,n)$, we define the following notion of limit set due to Chen and 
Greenberg, see \cite{CG}.
 
\begin{definition}
Let $\Gamma\subset PU(1,n)$, then $\Lambda_{CG}(\Gamma)$ is to be defined  as
$\overline{\Gamma x}\cap \partial \Bbb{H}^n_\Bbb{C}$, where $x\in \Bbb{H}^n_\Bbb{C}$ is any point.
\end{definition}

 As is the Fuchsian groups case, it is customary to show that $\Lambda_{CG}(\Gamma)$ does not depend on
 the choice of $x$ and $\Lambda_{CG}(\Gamma)$ has either 1,2 or infinite points.  A group is said to be
non-elementary if  $\Lambda_{CG}(\Gamma)$ has infinite points.
 In the following, given a projective subspace $P\subset \Bbb{P}^n_\Bbb{C}$ we will define 
\[
P^\bot=[\{w\in \Bbb{C}^{n+1}\mid \langle w,v\rangle=0 \textrm{ for all } v\in [P]^{-1} \}\setminus\{0\}].
\]

\begin{lemma} \label{l:dinlox}
Let $\gamma\in PU(1,n)$ be a loxodromic element and $a,r\in\partial\Bbb{H}^n_\Bbb{C}$ be the attracting and repelling points  of $\gamma$ respectively, thus
\begin{enumerate}
\item We have $\gamma^n \xymatrix{ \ar[r]_{m \rightarrow
\infty}&}  a$ uniformly on compact sets  of $\Bbb{P}^n_\Bbb{C}\setminus r^\bot$.

\item We have $\gamma^{-n} \xymatrix{ \ar[r]_{m \rightarrow
\infty}&}  r$ uniformly on compact sets  of $r^\bot\setminus a^{\bot}.$

\item  The transformation $\gamma$ restricted to $r^\bot\cap  a^{\bot}$ is conjugate to an element of $PU(n-1)$ acting on $\Bbb{P}^{n-1}_\Bbb{C}$
\end{enumerate} 
\end{lemma}

A complex hyperbolic manifold is the quotient of a open subset of the complex hyperbolic space and a 
discrete subgroup of $PU(1,n)$. Mok-Young  and Klingler  showed independently 
that for a complex hyperbolic manifold with finite volume there is a unique complex projective structure
compatible with  the complex hyperbolic structure, this  result will be crucial central in our discussion, 
see \cite{MY, klingler},  for sake  of completeness here we write down the result.

\begin{theorem}[Mok-Yeung, Klingler] \label{MYK}
Let $\Gamma\subset PU(1,n)$ be a discrete group such that  $M=\Bbb{H}^n_{\Bbb{C}}/\Gamma$ is  manifold
 of finite volume, then $M$ has only one projective structure compatible with the complex extructure.
\end{theorem}

This is very deep result which has several dynamical consequences  some of which  we will se 
later.

\section{Schottky like groups}\label{s:slg}
Schottky groups are the  ``simplest'' examples of Kleinian groups and enjoy very 
interesting properties, unfortunately they are always realizable in the higher dimensional setting (see \cite{cano}) 
in a ``usual'' way. For this reason let us introduce here a weaker form of  Schottky  groups, 
compare with \cite{CoG}.

\begin{definition} Let  $\Sigma\subset PSL(n+1,\Bbb{C})$ 
a finite set  which is symmetric ({\it i. e. } $a^{-1} \in\Sigma $ for all
$a\in \Sigma $) and $\{A_a\}_{a \in \Sigma }$ a family of  compact non-empty pairwise disjoint subsets of
 $\Bbb{P}^n_{\Bbb{C}}$ such that 
 for each  $a \in \Sigma$ we have  
  $$
  \bigcup_{b\in \Sigma\setminus\{a^{-1}\}}a(A_b) \varsubsetneq	  A_a.
  $$
The group $\Gamma$ generated  by $\Sigma$ 
is called  Schottky like group. We define a ``kind'' of limit set for $\Gamma$ as follows:
\[
\Lambda_{S}(\Gamma)=
\overline{
\{
y\in \Bbb{P}^n_{\Bbb{C}}\vert \exists (\phi_m)\subset \Sigma, (y_n)\subset A_{\phi_0}: \phi_{j+1}\phi_j\neq Id,  \phi_{m}\circ\ldots\circ \phi_1(y_m) \xymatrix{ \ar[r]_{m \rightarrow
\infty}&}  y
\}
}
\]

\end{definition}

Clearly every Schottky like group is a free, finitely generated and  discrete.

\begin{example}
Every Schottky group of $PSL(2,\Bbb{C})$ acting on $\Bbb{P}^1_{\Bbb{C}}$
is a Schottky like group.
\end{example}

The following is a straightforward result which will be used later, see \cite{cano, cl} for a proof.

\begin{lemma}\label{l:dinpar}
Let us consider the cyclic group $\Gamma\subset PSL(n+1,\Bbb{C})$ generated by the element
\[
\gamma=
\begin{bmatrix}
A\\
   &B
   \end{bmatrix}
\]
where $A$ is a $k\times k$ diagonalisable matrix with unitary proper values  and
$B$ is a $(n+1-k)\times (n+1-k)$-Jordan block whose proper value is $ Det(A)^{-(n+1-k)^{-1}}$. Then 
\begin{enumerate}
\item  If $x\in \Bbb{P}^n_\Bbb{C}\setminus Span(e_{1},\ldots e_{k})$, then the  set of accumulation points of 
$\Gamma x$ is $e_{1}$.
\item  If $x\in Span(e_{1},\ldots e_{k})$, then $ x$ belongs to the  set of accumulation points of  $\Gamma x$.
\end{enumerate}
\end{lemma}

\begin{lemma}\label{l:hsl}
Let $\Gamma\subset PU(1,n)$ be a Schottky like group, thus 

\begin{enumerate}
\item \label{i:1l} The group  $\Gamma$ is  purely loxodromic group.  

\item \label{i:2l} We have $$\Lambda_{CG}(\Gamma)\subset \Lambda_{S}(\Gamma).$$

\item \label{i:3l} The set  $\Lambda_{CG}(\Gamma)$ is disconnected. 
\end{enumerate}
\end{lemma}
\begin{proof}
Let us show (\ref{i:1l}). Clearly, it  will be enough to show that every  generator is loxodromic. Let 
$\gamma\in \Gamma$ be a generator, since $\Gamma$ is free, we deduce that $\gamma$ is either  parabolic
or loxodromic. Let us assume that $\gamma$ is parabolic, since 
$\gamma(A_\gamma) \varsubsetneq A_\gamma$ and 
$\gamma^{-1}(A_{\gamma^{-1}}) \varsubsetneq A_{\gamma^{-1}}$
for some pairwise disjoint, non-empty compact sets of $\Bbb{P}^{n}_{\Bbb{C}}$ we deduce that $\gamma$ 
has at least two fixed points. Therefore $\gamma$ has a lift $\widetilde \gamma\in SL(n+1,\Bbb{C})$ whose
normal Jordan form is, see \cite{CG, cl}:
\[
\gamma=
\begin{bmatrix}
A\\
   &B
   \end{bmatrix}
\]
where $A$ is a $k\times k$ diagonalisable matrix with unitary proper values  and $B$ is a 
$(n+1-k)\times (n+1-k)$-Jordan block whose  proper value is $Det(A)^{-(n+1-k)^{-1}}$.  Finally, let 
$x\in A_\gamma$ and $y\in  A_{\gamma^{-1}}$ be  fixed point, then by Lemma \ref{l:dinpar} we conclude
 that $x=y=e_{k+1}$, which is a contradiction.

The proof  of (\ref{i:2l}) goes as follows. Since $\Gamma$ is  free we deduce that $\Gamma$ is
non-elementary, see \cite{kamiya}. Therefore, it  will be enough to show that for every  generator 
$Fix(\gamma)\cap \partial \Bbb{H}^n_{C}\cap \Lambda_{S}(\Gamma)\neq \emptyset$. On the contrary, 
let us assume that  $Fix(\gamma)\cap \partial \Bbb{H}^n_{C}\cap \Lambda_{S}(\Gamma)= \emptyset$ 
for every generator. Let $\gamma_1,\gamma_2\in \Gamma$ be a generators satisfying 
$\gamma_1\gamma_2\neq Id$. For each $i=1,2$, let $A_{\gamma_i}$ be the generating set of $\gamma_i$,
since $\gamma_1$ is loxodromic we can consider $a,r\in \partial \Bbb{H}^n_{\Bbb{C}}$ the attracting and 
repelling fixed points of $\gamma_1$, respectively. Therefore $A_{\gamma_2}\subset a^\bot \cap r^\bot$, in 
consequence  $A_{\gamma_2}\subset  A_{\gamma_1}$. Since $\gamma_1$ restricted to $r^\bot \cap a^\bot$ 
is elliptic, which is a contradiction.

 Finally observe that part  (\ref{i:3l}) is trivial.
\end{proof}

\section{Proof of the main theorems}\label{s:pmt}

\subsection*{Proof of theorem \ref{t:1}}
Let $M$ be a compact complex manifold such that $M=\Bbb{H}^n_\Bbb{C}/\Gamma$, where  $n\geq 2$ 
and $\Gamma\subset PU(1,n)$ is a discrete group. Let us consider the  manifold $M\sqcup M$ and let us 
assume that there is  a  group $G\subset PSL(n+1,\Bbb{C})$ and a $G$-invariant  open set
 $U\subset \Bbb{P}^n_\Bbb{C}$  satisfying $M\sqcup M=U/G$. By   Theorem \ref{MYK}, we deduce that 
 $G=\Gamma$, up to projective conjugation. On the there hand, by the main theorem in \cite{CLL}, we know 
 that $\Bbb{H}^n_\Bbb{C}$ is the largest open set of $\Bbb{P}^n_{\Bbb{C}}$ on which $\Gamma$  acts 
 properly discontinuously, which is a contradiction.$\square$

\subsection*{Proof of theorem \ref{t:2}} [Geometric version]
Let $M$ be a compact complex manifold  such that $M=\Bbb{H}^n_\Bbb{C}/\Gamma$, where  $n\geq 2$ 
and $\Gamma\subset PU(1,n)$ is a discrete group.  Let us assume that there  is $G\subset PSL(n+1,\Bbb{C})$
 a Schottky like group and a $G$-invariant  open set $U\subset \Bbb{P}^n_\Bbb{C}$   $M=U/G$.  By  
 Theorem   \ref{MYK}  we deduce that $G=\Gamma$, up to projective conjugation.
Finally observe that $\Lambda_{CG}(\Gamma)=\partial\Bbb{H}^{n}_{\Bbb{C}}$, since $M$ is compact, 
however, this contradicts  Lemma \ref{l:hsl}.$\square$

\subsection*{Proof of theorem \ref{t:2}}  [Algebraic version]
Let $M$ be a compact complex manifold  such that $M=\Bbb{H}^n_\Bbb{C}/\Gamma$, where  $n\geq 2$
 and $\Gamma\subset PU(1,n)$ is a discrete group.  Let us assume that there  is   
 $G\subset PSL(n+1,\Bbb{C})$ a  discrete, purely loxodromic free  group and a
$G$-invariant  open set $U\subset \Bbb{P}^n_\Bbb{C}$  satisfying $M=U/G$.  
By  Theorem   \ref{MYK}  we deduce that $G=\Gamma$, up to projective conjugation. Since $M$ is compact,
 we get that the Cayley graph $\Delta (\Gamma)$ of $\Gamma$ is quasi-isometric to $\Bbb{H}^n_\Bbb{C}$, 
 see \cite{bow}, in consequence the Gromov boundaries $\partial \Delta (\Gamma)$ and
  $\partial \Bbb{H}^n_\Bbb{C}$ are homeomorphic, see \cite{bow}. This is a contradiction since is well known,
   see \cite{bow}, that  $\partial \Delta (\Gamma)$ is a Cantor set while  $\partial \Bbb{H}^n_\Bbb{C}$ is the
    $2n-1$-sphere.$\square$

As we have seen, the proofs of the results are in terms of compacts complex 
hyperbolic manifolds a more interesting question  is to determine 
sufficient conditions for the uniformization theorems remain valid in the higher 
dimensional case, see for example \cite{klingler1} for results concerning compact comples projective surfaces.

\section*{Acknowledgments}
The authors would like to thank to W. Barrera, J. P. Navarrete, P. Py,   J. Seade  \&  A. Verjovsky   for 
fruitful conversations.  Last but not the least, we want to thank to the people 
and staff of the UCIM for all his help during the elaboration of this article.
\bibliographystyle{amsplain} 
\bibliography{mybib}

\end{document}